\documentclass[12pt]{article}

\usepackage[T1]{fontenc}
\usepackage[utf8]{inputenc}
\usepackage[english]{babel}
\usepackage{courier}
\usepackage{graphicx}
\usepackage{float}
\usepackage{amsthm}
\usepackage{amsmath}
\usepackage{amssymb}
\usepackage{concrete}
\usepackage{geometry}
\usepackage{caption}
\usepackage{subcaption}
\usepackage{algorithm,algpseudocode}
\usepackage{tikz}
\usepackage{wrapfig}
\usepackage{multirow}
\usepackage{hyperref}
\usepackage{sidecap}
\usepackage{authblk}
\usetikzlibrary{decorations.pathreplacing}
\usepackage{multicol}

\theoremstyle{plain}
\newtheorem{theorem}{Theorem}

\theoremstyle{definition}

\theoremstyle{remark}
\newtheorem{cons}{Corollary}
\newtheorem{remark}{Remark}

\newcommand{\rank}{\mathrm{rank}\kern 2pt}

\newcommand{\cV}{{\rm vol}}

\algrenewcommand\algorithmiccomment[1]{\{#1\}}
\algnewcommand{\To}{\textbf{to}\ }

\floatname{algorithm}{Algorithm}

\renewcommand{\ref}{\href}

\newcommand*{\myfont}{\fontfamily{phv}\selectfont}

\begin{document}

\title{Close to optimal column approximations with a single SVD}

\author[1]{Osinsky~A.I.}

\affil[1]{Skolkovo Institute of Science and Technology, Moscow, Russia}

\maketitle

\begin{abstract}
The best column approximation in the Frobenius norm with $r$ columns has an error at most $\sqrt{r+1}$ times larger than the truncated singular value decomposition. Reaching this bound in practice involves either expensive random volume sampling or at least $r$ executions of singular value decomposition. In this paper it will be shown that the same column approximation bound can be reached with only a single SVD (which can also be replaced with approximate SVD). As a corollary, it will be shown how to find a highly nondegenerate submatrix in $r$ rows of size $N$ in just $O(Nr^2)$ operations, which mostly has the same properties as the maximum volume submatrix.
\end{abstract}

\section{Introduction}

This paper studies how to construct a good approximation of rank $r$ of the matrix $A \in \mathbb{C}^{M \times N}$ based on $r$ columns $C \in \mathbb{C}^{M\times r}$ (column approximation) or columns $C$, rows $R \in \mathbb{C}^{r\times N}$ and the submatrices $\hat A \in \mathbb{C}^{r \times r}$ at their intersection (skeleton approximation).

The task of finding optimal columns is often discussed in data analysis. Usually, the following problem is considered, called the Column Subset Selection Problem (CSS/CSSP) \cite{CCAlargecoef,CCAsvd}:
\[
  \left\| A - CC^+ A \right\|_F \to \min\limits_C.
\]
Note, however, that choosing optimal columns is an NP-complete problem \cite{ColumnHard}, although it is possible to speed up the exhaustive search using A*-like algorithm \cite{CCAstar}. Other popular algorithms include greedy search \cite{GreedyCCA3} and iterative update of the selected columns one by one \cite{CCAiter}.

In \cite{bestCW} it was shown that it is possible to construct a $CC^+A$ approximation of rank $r$ with an error coefficient in the Frobenius norm not greater than $\sqrt{r+1}$:
\begin{equation}\label{eq:colopt}
  \left\| A - CC^+A \right\|_F \leqslant \sqrt{r+1} \left\| A - A_r \right\|_F,
\end{equation}
where the error is estimated in terms of the truncated singular value decomposition $A_r$ of the matrix $A$, $\rank A_r = r$.

It was also proved in \cite{bestCW} that the coefficient $\sqrt{r+1}$ is optimal and cannot be reduced. However, the best known algorithm of finding columns (not necessarily optimal) satisfying \eqref{eq:colopt} requires $O\left(r M^2 N\right)$ operations \cite{Alice2} (authors also show there how a good skeleton approximation can be constructed with the same complexity bound). The same number of operations is required by a randomized algorithm based on averaging over the squared volume of the columns $C$, chosen randomly \cite{bestCW}, where the expected error also does not exceed $\sqrt{r+1} \left\| A - A_r \right\|_F$.

In both cases, the cost of the algorithms is excessively high. Here we will construct a method for obtaining an estimate with a coefficient $\sqrt{r+1}$ using only one singular decomposition, and therefore having full complexity $O\left( MN \min\left(M, N \right) \right)$. In addition, we will show how in $O \left( MNr \right)$ operations an arbitrary approximation $Z, \rank Z = r$ of the matrix $A$ can be turned into a column approximation with an error of at most $\sqrt{r+1} \left\|A - Z\right\|_F$ and into a skeleton decomposition with an error at most $(r+1) \left\|A - Z\right\|_F$, which coincides with the best known estimate from \cite{me-fnorm}.

\section{Construction of close to optimal column approximation}

\begin{theorem}\label{th:fastbest}
Let $A, Z \in \mathbb{C}^{M \times N}$, $\rank Z= r$ be given. Then in $O \left( MNr \right)$ operations it is possible to find columns $C \in \mathbb{C}^{M \times r}$ of the matrix $A$ and weights $W \in \mathbb{C}^{r \times N}$ based on the matrix $Z$, such that simultaneously
\begin{equation}\label{eq:cwf}
  \left\| A - C C^+ A \right\|_F \leqslant \left\| A - CW \right\|_F \leqslant \sqrt{ r + 1 } \left\| A - Z \right\|_F
\end{equation}
and
\begin{equation}\label{eq:cw2}
  \left\| A - CC^+A \right\|_2^2 \leqslant \left\| A - CW \right\|_2^2 \leqslant \| A - Z\|_2^2 + r \left\| A - Z \right\|_F^2 \leqslant \left( 1 + r \left( \min \left(M, N \right) - r \right) \right) \| A - Z \|_2^2.
\end{equation}
\end{theorem}
\begin{proof}
Let us use the rows $V \in \mathbb{C}^{r \times N}$ of the matrix of the right singular vectors from the truncated singular value decomposition $Z = U \Sigma V$ and orthogonalize the matrix $A$ to them:
\[
  \tilde A = A - A V^* V.
\]
Since $Z - Z V^* V = 0$,
\begin{equation}\label{eq:tildea}
  \left\| \tilde A \right\|_{2,F} = \left\| A - A V^* V \right\|_{2,F} = \left\| \left( A - Z \right) \left(I - V^* V \right) \right\|_{2,F} \leqslant \left\| A - Z \right\|_{2,F}.
\end{equation}

Hereinafter we simultaneously use both lower indices $2$ and $F$, when the expression is valid both in the spectral norm and in the Frobenius norm.

We will build our column approximation using the weights $W = \hat V^{-1} V$, where the submatrix $\hat V \in \mathbb{C}^{r \times r}$ corresponds to the columns $C$. If we denote by $\tilde C = C - A V^*\hat V$ the corresponding columns of $\tilde A$, then
\[
\begin{aligned}
  \left\| A - CW \right\|_{2,F} & = \left\| \tilde A + A V^* V - \tilde C W - A V^* \hat V W \right\|_{2,F} \\
  & = \left\| \tilde A + A V^* V - \tilde C W - A V^* \hat V \hat V^{-1} V \right\|_{2,F} \\
  & = \left\| \tilde A - \tilde C W \right\|_{2,F}.
\end{aligned}
\]
Note that for now calculating $\tilde A$ required only $O \left( MNr \right)$ operations, since the singular value decomposition of a rank-$r$ matrix can be computed quickly.

Without loss of generality, we can assume that we will eventually select the first $r$ columns. Next we will describe how exactly to choose them. Note also that the weights $W$ and $\tilde A$ remain unchanged when $V$ is replaced by any $V^Q = QV$, where $Q \in \mathbb{C}^{r \times r}$ is an arbitrary unitary matrix.

Let us choose $Q$ so the matrix $\hat V^Q$ is upper triangular. We are going to add the columns one by one, which will allow us to prove the estimate by induction. At any step $k \geqslant 1$ (without loss of generality, we assume that the first $k-1$ columns have been selected), the first $k-1$ columns of $\tilde A$ will be zero:
\[
  \tilde C - \tilde C \hat W = \tilde C - \tilde C \hat V^+ \hat V = 0.
\]

If we denote
\[
  \tilde A^{(k-1)} = \tilde A - C^{(k-1)} W^{(k-1)}, \quad C^{(k-1)} \in \mathbb{C}^{M \times (k-1)}, \; W^{(k-1)} \in \mathbb{C}^{(k-1) \times N}.
\]
then (with $i$:$j$ we denote choosing rows/columns from $i$ to $j$) approximation at the next step is as follows:
\[
\begin{aligned}
  \tilde A^{(k)} & = \tilde A - C^{(k)} W^{(k)} \\
  & = \tilde A - C^{(k)} \hat V^{(k),Q,-1} V_{1:k,1:N}^Q \\
  & = \tilde A - \left[ \tilde A_{1:M,1:k-1} \; \tilde A_{1:M,k} \right] \left[ {\begin{array}{*{20}{c}}
   \hat V^{(k-1),Q,-1} & -\hat V^{(k-1),Q,-1} V_{1:k-1,k}^Q  / \hat V_{kk}^Q \\ 
   0 & 1 / \hat V_{kk}^Q  \\ 
\end{array} } \right] V_{1:k,1:N}^Q \\
  & = \tilde A - C^{(k-1)} \hat V^{(k-1),Q,-1} V_{1:k-1,1:N}^Q + C^{(k-1)} \hat V^{(k-1),Q,-1} V_{1:k-1,k}^Q \frac{V_{k,1:N}^Q}{\hat V_{kk}^Q} - \tilde A_{1:M,k} \frac{V_{k,1:N}^Q}{\hat V_{kk}^Q} \\
  & = \tilde A - C^{(k-1)} W^{(k-1)} + C^{(k-1)} W_{1:k-1,k}^{(k-1)} \frac{V_{k,1:N}^Q}{\hat V_{kk}^Q} - \tilde A_{1:M,k} \frac{V_{k,1:N}^Q}{\hat V_{kk}^Q} \\
  & = \tilde A - C^{(k-1)} W^{(k-1)} - \tilde A_{1:M,k}^{(k-1)} \frac{V_{k,1:N}^Q}{\hat V_{kk}^Q} \\
  & = \tilde A^{(k-1)} - \frac{1}{\hat V_{kk}^Q} \tilde A_{1:M,k}^{(k-1)} V_{k,1:N}^Q.
\end{aligned}
\]
In other words, the error at the $k$-th step $\tilde A^{(k)}$ is the approximation error of $\tilde A^{(k-1)}$ using its $k$-th column $\tilde A_{1:M,k}^{(k-1)}$ and has the following form:
\begin{equation}\label{eq:akbasic}
  \tilde A^{(k)} = \tilde A^{(k-1)} - \frac{1}{\hat V_{kk}^Q} \tilde A_{1:M,k}^{(k-1)} V_{k,1:N}^Q.
\end{equation}
Denote by $\bar V^Q \in \mathbb{C}^{(r-k+1) \times N}$ the last $r-k+1$ rows of $V^Q$. Then \eqref{eq:akbasic} can be written as
\[
  \tilde A^{(k)} = \tilde A^{(k-1)} - \tilde A_{1:M,k}^{(k-1)} \bar V_{1:r-k+1,k}^{Q,+} \bar V^Q.
\]

Now we greedily choose the column $j$ (which will take place of the $k$-th column), which minimizes the Frobenius norm error at the $k$-th step. To do that, we can use the condition
\begin{equation}\label{eq:jopt}
  j = \mathop{\arg\min}\limits_{j \geqslant k} \left\| \tilde A_{1:M,j}^{(k-1)} \right\|_2 / \| \bar V_{1:r-k+1,j}^Q \|_2.
\end{equation}

Since the first $k-1$ columns of $V^Q$ contain $(k-1) \times (k-1)$ right triangular submatrix, we can rewrite the denominator in \eqref{eq:jopt} as
\[
\begin{aligned}
  \| \bar V_{1:r-k+1,j}^Q \|_2 & = \| V_{k:r,j}^Q \|_2 \\
  & = \| \left( I - V_{1:r,1:k-1}^Q V_{1:r,1:k-1}^{Q,+} \right) V_{1:r,j}^Q \|_2 \\
  & = \| \left( I - Q V_{1:r,1:k-1} V_{1:r,1:k-1}^+ Q^* \right) Q V_{1:r,j} \|_2 \\
  & = \| Q \left( I - V_{1:r,1:k-1} V_{1:r,1:k-1}^+ \right) V_{1:r,j} \|_2 \\
  & = \| \left( I - V_{1:r,1:k-1} V_{1:r,1:k-1}^+ \right) V_{1:r,j} \|_2
\end{aligned}
\]
to show that it does not actually depend on yet unknown transformation $Q$. In practice, transformation $Q$ can be constructed using Householder reflections, obtained after each new column is added.

Now note that
\begin{equation}\label{eq:sum1}
\sum\limits_{j=1}^N \left\| \tilde A_{1:M,j}^{(k-1)} \right\|_2^2 = \left\| \tilde A^{(k-1)} \right\|_F^2
\end{equation}
and
\begin{equation}\label{eq:sum2}
\sum\limits_{j=1}^N \left\| \bar V_{1:r-k+1,j}^Q \right\|_2^2 = \left\| \bar V^Q \right\|_F^2 = r - k + 1.
\end{equation}
We then use the general inequality
\begin{equation}\label{eq:avg}
  \frac{\sum\limits_{j = 1}^N \alpha_j}{\sum\limits_{j = 1}^N \beta_j} \geqslant \frac{\sum\limits_{j = 1}^N \min\limits_k \frac{\alpha_k}{\beta_k} \beta_j }{\sum\limits_{j = 1}^N \beta_j} = \min\limits_k \frac{\alpha_k}{\beta_k},
\end{equation}
which holds whenever $\sum\limits_j \beta_j$ is positive and all $\alpha_j$ are nonnegative. In our case, for the chosen column $j$, which minimizes \eqref{eq:jopt}, we can substitute the ratio of \eqref{eq:sum1} and \eqref{eq:sum2} to estimate the minimum from above:
\[
  \left\| \tilde A_{1:M,k}^{(k-1)} \bar V_{1:r-k+1,k}^{Q,+} \bar V^Q \right\|_2^2 = \left\| \tilde A_{1:M,k}^{(k-1)} \bar V_{1:r-k+1,k}^{Q,+} \right\|_2^2 = \frac{\left\| \tilde A_{1:M,k}^{(k-1)} \right\|_2^2}{\left\| \bar V_{1:r-k+1,k}^{Q} \right\|_2^2} \leqslant \frac{\left\| \tilde A^{(k-1)} \right\|_F^2}{r - k + 1}.
\]

The rows of $\tilde A{(k-1)}$ are orthogonal to $\bar V^Q$:
\[
  \tilde A^{(k-1)} \bar V^{Q,*} = \tilde A V_{k:r,1:N}^{Q,*} - C^{(k-1)} \hat V^{(k),Q,-1} V_{1:k-1,1:N}^Q V_{k:r,1:N}^{Q,*} = 0 - 0 = 0,
\]
since the rows $V^Q$ are orthogonal to each other, and the rows of $\tilde A$ are orthogonal to $V$ by construction. Using this fact, the error at the next step can be estimated as
\[
\begin{aligned}
  \left\| \tilde A^{(k-1)} - \tilde A_{1:M,k}^{(k-1)} \bar V_{1:r-k+1,k}^{Q,+} \bar V^Q \right\|_{2,F}^2 & \leqslant \left\| \tilde A^{(k-1)} \right\|_{2,F}^2 + \left\| \tilde A_{1:M,k}^{(k-1)} \bar V_{1:r-k+1,k}^{Q,+} \bar V^Q \right\|_2^2 \\
  & = \left\| \tilde A^{(k-1)} \right\|_{2,F}^2 + \left\| \tilde A_{1:M,k}^{(k-1)} \right\|_2^2 / \left\| \bar V_{1:r-k+1,k}^{Q} \right\|_2^2 \\
  & \leqslant \left\| \tilde A^{(k-1)} \right\|_{2,F}^2 + \frac{1}{r-k+1} \left\| \tilde A^{(k-1)} \right\|_{F}^2.
\end{aligned}
\]
In the Frobenius norm the first inequality is actually equality, which means the choice of $j$ \eqref{eq:jopt} indeed minimizes the Frobenius norm.

We conclude, that the error at the next step is determined by the error in the previous step. Substituting previous $k-1$ steps, we get
\[
  \left\| \tilde A^{(k)} \right\|_{2,F}^2 = \left\| \tilde A^{(k-1)} - \tilde A_{1:M,k}^{(k-1)} \bar V_{1:r-k+1,k}^{Q,+} \bar V^Q \right\|_{2,F}^2 \leqslant \left\| \tilde A \right\|_{2,F}^2 + \frac{k}{r-k+1} \left\| \tilde A \right\|_{F}^2.
\]
For $k = r$, using \eqref{eq:tildea}, we obtain \eqref{eq:cwf} and \eqref{eq:cw2}. Each step of the algorithm requires $O \left( MN \right)$ operations to go from $A^{(k-1)}$ to $A^{(k)}$ and to find the optimal $j$ \eqref{eq:jopt}, so the total cost is $O \left( MNr \right)$. To get the second part of \eqref{eq:cw2}, we can use the inequality
\begin{equation}\label{eq:tildeaf2}
  \left\| \tilde A \right\|_F^2 \leqslant \rank \tilde A \cdot \left\| \tilde A \right\|_2^2 \leqslant \left( \min \left(M, N \right) - r \right) \left\| \tilde A \right\|_2^2 \leqslant \left( \min \left(M, N \right) - r \right) \left\| A - Z \right\|_2^2.
\end{equation}

Inequalities for $CC^+A$ approximation follow from the fact that the choice $W = C^+A$ minimizes the error for any fixed columns $C$:
\[
\begin{aligned}
  \left\| A - CW \right\|_{2,F} & = \left\| CC^+ \left( A - CW \right) + \left( I - CC^+ \right) \left( A - CW \right) \right\|_{2,F} \\
  & \geqslant \left\| \left( I - CC^+ \right) \left( A - CW \right) \right\|_{2,F} \\
  & = \left\| A - CC^+ A + CW - CW \right\|_{2,F} \\
  & = \left\| A - CC^+ A \right\|_{2,F}.
\end{aligned}
\]
\end{proof}
\begin{cons}
If $Z = A_r$, where $A_r$ is the truncated singular value decomposition of $A$, then we get the bound
\[
  \left\| A - CC^+ A \right\|_F \leqslant \sqrt{r+1} \left\|A - A_r \right\|_F,
\]
which coincides with the best bound from \cite{bestCW}. And, as we mentioned earlier, the coefficient $\sqrt{r+1}$ cannot be improved. Constructing such an approximation requires one singular value decomposition and therefore has $O\left(\min\left(M, N\right) M N \right)$ complexity.
\end{cons}
\begin{remark}
Derandomization of volume sampling from \cite{Alice2}, which speeds up the algorithm from \cite{CWexist}, requires $O \left( r M^2 N \right)$ operations and guarantees
\[
  \left\| A - CC^+ A \right\|_F \leqslant \sqrt{r+1} \sqrt{\frac{\sum\limits_{i_1 < \ldots < i_{r+1}} \sigma_{i_1}^2 \left( A \right) \cdot \ldots \cdot \sigma_{i_{r+1}}^2 \left( A \right)}{\sum\limits_{i_1 < \ldots < i_r} \sigma_{i_1}^2 \left( A \right) \cdot \ldots \cdot \sigma_{i_r}^2 \left( A \right)}},
\]
where the right-hand-side contains sums of products of squares of different singular values of $A$ from all sets of size $r+1$ (in the numerator) and size $r$ (in the denominator). This bound is never higher than $\sqrt{r+1} \left\|A - A_r \right\|_F$, which is how the latter bound was first obtained in \cite{CWexist}. The coefficient also becomes close to 1, when more than $r$ columns are selected \cite{newCCA}. The bound of derandomized volume sampling coincides with ours in the limit $\sigma_r \left( A \right) / \sigma_{r+1} \left( A \right) \to \infty$.
\end{remark}

The algorithm, which selects $r$ columns, is based on the criterion \eqref{eq:jopt} and the update \eqref{eq:akbasic}. It is written as Algorithm \ref{alg:CCA}. The current right triangular submatrix is obtained using Householder reflections.

\begin{algorithm}[!ht]
\caption{}
\label{alg:CCA}
\begin{algorithmic}[1]
\Require{Matrix $A \in \mathbb{C}^{M \times N}$, approximation $Z \in \mathbb{C}^{M \times N}$ of rank $r$.}
\Ensure{Columns $C \in \mathbb{C}^{M \times r}$ and weights $W \in \mathbb{C}^{r \times N}$, which satisfy Theorem \ref{th:fastbest}.}
  \State Calculate $V \in \mathbb{C}^{r \times N}$ -- right singular vectors of $Z$
  \State $\tilde A := A - A V^* V$
\For{$k := 1$ \To $r$}
  \State $j = \mathop{\arg\min} \limits_{j \geqslant k} \left\| \tilde A_{1:M,j} \right\|_2 / \| V_{k:r,j} \|_2$
  \State Swap columns $k$ and $j$ in $\tilde A$ and in $V$
  \State $v := V_{k:r,k}$
  \State $v_1 := v_1 + e^{i \arg v_1} \left\| v \right\|_2$
  \State $v := v / \left\| v \right\|_2$
  \State $V_{k:r,1:N} := V_{k:r,1:N} - 2 vv^* V_{k:r,1:N}$
  \State $\tilde A := \tilde A - \frac{1}{V_{kk}} \tilde A_{1:M,k} V_{k,1:N}$
\EndFor
\State $W = V_{1:r,1:r}^{-1} V$
\end{algorithmic}
\end{algorithm}

Due to the arbitrariness of the initial approximation $Z$, we can apply Theorem \ref{th:fastbest} twice to obtain a skeleton approximation.

\begin{theorem}\label{th:fastcross}
Let $A, Z \in \mathbb{C}^{M \times N}$, $\rank Z = r$ be given. Then in $O \left( MNr \right)$ operations it is possible to find rows $R \in \mathbb{C}^{r \times N}$ and columns $C \in \mathbb{C}^{M \times r}$ of the matrix $A$, such that simultaneously
\begin{equation}\label{eq:ccarrf}
  \left\| A - C C^+ A R^+ R \right\|_F \leqslant \sqrt{2r + 2} \left\| A - Z \right\|_F
\end{equation}
and
\begin{equation}\label{eq:ccarr2}
  \left\| A - C C^+ A R^+ R \right\|_2 \leqslant \sqrt{2 + 2r \left( \min \left( M, N \right) - r \right)} \| A - Z \|_2.
\end{equation}
In the same number of operations it is also possible to select the same rows $R \in \mathbb{C}^{r \times N}$, but maybe different columns $C \in \mathbb{C}^{M \times r}$, such that simultaneously
\begin{equation}\label{eq:carf}
  \left\| A - C \hat A^{-1} R \right\|_F \leqslant \left( r + 1 \right) \left\| A - Z \right\|_F
\end{equation}
and
\begin{equation}\label{eq:car2}
  \left\| A - C \hat A^{-1} R \right\|_2 \leqslant \sqrt{  1 + r \left( r + 2 \right) \left( \min \left(M, N \right) - r \right) } \| A - Z \|_2,
\end{equation}
where $\hat A \in \mathbb{C}^{r \times r}$ is the submatrix at the intersection of rows $R$ and columns $C$.
\end{theorem}

\begin{proof}
First, let us prove the inequalities \eqref{eq:ccarrf} and \eqref{eq:ccarr2}:
\begin{equation}\label{eq:ccarrproof}
\begin{aligned}
  \left\| A - C C^+ A R^+ R \right\|_{2,F}^2 & = \left\| A - CC^+ A + CC^+ A - C C^+ A R^+ R \right\|_{2,F}^2 \\
  & \leqslant \left\| A - CC^+ A \right\|_{2,F}^2 + \left\| CC^+ \left( A - A R^+ R \right) \right\|_{2,F}^2 \\
  & \leqslant \left\| A - CC^+ A \right\|_{2,F}^2 + \left\| A - A R^+ R \right\|_{2,F}^2.
\end{aligned}
\end{equation}

To choose the columns $C$ we can directly use Theorem \ref{th:fastbest}. To select the rows, we can apply Theorem \ref{th:fastbest} to $A^T$. Substituting the corresponding estimates into \eqref{eq:ccarrproof}, we obtain \eqref{eq:ccarrf} and \eqref{eq:ccarr2}.

The inequality \eqref{eq:carf} can be proved in the same way as the column case is generalized to the cross case in \cite{meavg}. Namely, we can introduce a row approximation $\Phi = U \hat U^{-1}R$, where $U \in \mathbb{C}^{r \times N}$ is the matrix of the left singular vectors of $Z$. For the row approximation the estimates of the Theorem \ref{th:fastbest} are also valid. In particular, we obtain
\[
  \left\| A - \Phi \right\|_F \leqslant \sqrt{r+1} \left\| A - Z \right\|_F
\]
and (using an analogue of \eqref{eq:tildeaf2})
\[
  \left\| A - \Phi \right\|_F \leqslant \sqrt{\left( r+1 \right) \left( \min \left( M, N \right) - r \right)} \left\| A - Z \right\|_2.
\]

On the other hand, using Theorem \ref{th:fastbest} now for $Z = \Phi$, we get, using \eqref{eq:cw2},
\begin{equation}\label{eq:acvvf}
  \left\| A - C \hat V_{\Phi}^{-1} V_{\Phi} \right\|_F \leqslant \sqrt{r+1} \left\| A - \Phi \right\|_F \leqslant \left( r + 1 \right) \left\| A - Z \right\|_F
\end{equation}
and
\begin{equation}\label{eq:acvv}
\begin{aligned}
  \left\| A - C \hat V_{\Phi}^{-1} V_{\Phi} \right\|_2^2 & \leqslant \left\| A - \Phi \right\|_2^2 + r \left\| A - \Phi \right\|_F^2 \\
  & \leqslant \left(1 + r \left( \min\left(M, N \right) - r \right)\right) \left\| A - Z \right\|_2^2 + r \left( r + 1 \right) \left( \min\left(M, N \right) - r \right) \left\| A - Z \right\|_2^2 \\
  & = \left(1 + r \left( r + 2 \right) \left( \min\left(M, N \right) - r \right)\right) \left\| A - Z \right\|_2^2,
\end{aligned}
\end{equation}
where $V_{\Phi} \in \mathbb{C}^{r \times N}$ is the submatrix of the first $r$ right singular vectors of $U \hat U^{-1} R = \Phi = U_{\Phi} \Sigma_{\Phi} V_{\Phi}$. Like before, $\;\hat{}\;$ denotes the submatrix, corresponding to rows $R$ and/or columns $C$.

Now we only need to notice, that
\[
  \hat V_{\Phi}^{-1} V_{\Phi} = \left( \hat U_{\Phi} \Sigma_{\Phi} \hat V_{\Phi} \right)^{-1} \hat U_{\Phi} \Sigma_{\Phi} V_{\Phi} = \left( \hat U \hat U^{-1} \hat A \right)^{-1} \hat U \hat U^{-1} R = \hat A^{-1} R,
\]
so
\[
  A - C \hat V_{\Phi}^{-1} V_{\Phi} = A - C \hat A^{-1} R,
\]
and \eqref{eq:carf} and \eqref{eq:car2} follow from \eqref{eq:acvvf} and \eqref{eq:acvv}.
\end{proof}
\begin{cons}
Using $Z = A_r$, we get
\begin{equation}\label{eq:caropt}
  \left\| A - C \hat A^{-1} R \right\|_F \leqslant \left( r+ 1 \right) \left\| A - A_r \right\|_F,
\end{equation}
which coincides with the best known estimate for skeleton approximations of this form from \cite{me-fnorm}. Previously this bound was proved using averaging over submatrices $\hat A$, where each submatrix is chosen with the probability, proportional to its squared volume. In \cite{Alice2} it was shown how the submatrix, which satisfies \eqref{eq:caropt}, can be found in $O \left( M^2 N r \right)$ operations. Theorem \ref{th:fastcross} again requires the knowledge of $Z = A_r$ to get the best bound, which leads to the total complexity $O \left( M N \min \left( M, N \right) \right)$.
\end{cons}
\begin{remark}
As far as we know, the estimate \eqref{eq:car2} is the best spectral norm estimate for approximations of the form $C \hat A^{-1} R$. Estimates for $CUR$ approximations with $U = \hat A^{-1}$ are especially important to study, since approximations of this kind are constructed using rank-revealing LU factorizations \cite{strongRRLU,Pan,merholoc}, adaptive cross approximation \cite{BebAlg} and by using the {\myfont maxvol} algorithm \cite{maxvol}.
\end{remark}

Finally, it is worth noting that the estimate for the spectral norm similar to \eqref{eq:car2} can be obtained faster if we use (for $M<N$) $\Phi^T = QR, \; Q\in \mathbb{C}^{M\times r}, \;R \in\mathbb{C}^{r\times N}$ from an incomplete $QR$ factorization based on columns corresponding to a submatrix of locally maximum volume. The algorithm for finding such factorization was proposed in \cite{rrqr}. Here we will need an estimate for the spectral norm error of the incomplete $QR$ factorization.

\begin{theorem}[\cite{rrqr}]\label{th:rrqr}
In $O \left( MNr \left( 1 + \log_{\rho} r \right) \right)$ operations it is possible to construct an incomplete $QR$ factorization of $A = QR \in \mathbb{C}^{M \times N}, \; Q \in \mathbb{C}^{M \times r}, \; R \in \mathbb{C}^{r \times N}$, such that
\[
  \left\| A - QR \right\|_2 \leqslant \sqrt{1 + \rho^2 r \left( N - r \right)} \left\| A - A_r \right\|_2.
\]
\end{theorem}

The corresponding complexity bound was obtained in \cite{mearxiv,dominant}.

Applying Theorem \ref{th:rrqr} for $M \times (r+1)$ matrices that are extensions of $C$ by one column (when approximating $A^T$), we get in total for all columns
\[
  \left\| A^T - QR \right\|_F^2 \leqslant \left( 1 + \rho^2 r \left( r + 1 - r \right) \right) \left( M - r \right) \left\| A - A_r \right\|_2^2 \leqslant \left( 1 + \rho^2 r \right) \left( M - r \right) \left\| A - A_r \right\|_2^2.
\]

Similar to the proof of \eqref{eq:acvv}, using $\Phi^T = CW = QR$ to approximate $A^T$ leads to
\[
  \left\| A - C \hat A^{-1} R \right\|_2^2 = \left\| A - C \hat V_{\Phi}^{-1} V_{\Phi} \right\|_2^2 \leqslant \left( 1 + r \left( \rho^2 r + \rho^2 + 1 \right) \left( M - r \right) \right) \left\| A - A_r \right\|_2^2.
\]

Since we assumed $M < N$ (otherwise we can start from approximating $A$ instead of $A^T$), we have just proven the following result.

\begin{theorem}\label{th:fast2norm}
In $O \left( MNr \left( 1 + \log_{\rho} r \right) \right)$ operations it is possible to construct a skeleton approximation of $A \in \mathbb{C}^{M \times N}$ based on rows $R \in \mathbb{C}^{r \times N}$, columns $C \in \mathbb{C}^{M \times r}$ and the submatrix at their intersection $\hat A \in \mathbb{C}^{r \times r}$, such that
\[
  \left\| A - C \hat A^{-1} R \right\|_2 \leqslant \sqrt{1 + r \left( \rho^2 r + \rho^2 + 1 \right) \left( \min \left( M, N \right) - r \right)} \left\| A - A_r \right\|_2.
\]
\end{theorem}

Usually, in order to guarantee the accuracy of the skeleton approximation in the spectral norm, algorithms search for a submatrix $\hat A$ of locally maximum volume \cite{strongRRLU,Pan,merholoc}. However, such an approach can lead to an error that is about $\sqrt{MN}$ larger than the optimal.

Consider a simple example:
\[
  A = \left[ {\begin{array}{*{20}{c}}
   {1 + \varepsilon } & 1 &  \cdots  & 1  \\ 
   1 & 1 &  \cdots  & 1  \\ 
    \vdots  &  \vdots  &  \ddots  &  \vdots   \\ 
   1 & 1 &  \cdots  & 1  \\ 
\end{array} } \right] \in \mathbb{R}^{N \times N},
\]
Let us look for a good rank 1 skeleton approximation. Algorithms from \cite{strongRRLU,Pan,merholoc} are going to choose the maximum volume submatrix $\hat A = A_{11}$. The error of such approximation is going to be
\[
  \left\| A - A_{:,1} A_{11}^{-1} A_{1,:} \right\|_2 = \left( N - 1 \right) \left( 1 - \frac{1}{1 + \varepsilon} \right),
\]
since each element outside the first row and column is going to be $1 - \frac{1}{1 + \varepsilon}$. Therefore, the error grows linearly with matrix size.

If we use the first row (which has maximum volume), but then choose the column, which minimizes Frobenius norm error, then we get
\[
  \left\| A - A_{:,2} A_{12}^{-1} A_{1,:} \right\|_2 = \varepsilon \sqrt{N-1},
\]
since the error is going to be nonzero only in the first column. This approximation can be achieved using the algorithm from Theorem \ref{th:fast2norm}.

Finally, if we know the approximation
\[
  Z = \left[ {\begin{array}{*{20}{c}}
   1 & 1 &  \cdots  & 1  \\ 
   1 & 1 &  \cdots  & 1  \\ 
    \vdots  &  \vdots  &  \ddots  &  \vdots   \\ 
   1 & 1 &  \cdots  & 1  \\ 
\end{array} } \right] \in \mathbb{R}^{N \times N},
\]
in advance, then using Theorem \ref{th:fastcross} we can select the second row and the second column, which leads to
\[
  \left\| A - A_{:,2} A_{22}^{-1} A_{2,:} \right\|_2 = \varepsilon,
\]
but to achieve it we needed to know the approximation $Z$, which is close to truncated SVD.

\section{Quick search for a highly nondegenerate submatrix}

The column search method from Theorem \ref{th:fastbest} can also be used to search for highly nondegenerate submatrices. In particular, let us be given the rows $V \in \mathbb{C}^{r \times N}, \; VV^* = I$ and we need to find a submatrix $\hat V \in \mathbb{C}^{r \times r}$, such that $\left\| \hat V^{-1} \right\|_2$ is limited. In \cite{trn_min} it is shown that the maximum volume submatrix, $\cV \left( \hat V \right) = \left| \det \hat V \right|$, the following inequality holds
\begin{equation}\label{eq:maxvol2norm}
  \left\| \hat V^{-1} \right\|_2 \leqslant \sqrt{1 + r \left( N-r \right)}.
\end{equation}
The same proof in \cite{trn_min} also leads to the following inequality
\begin{equation}\label{eq:maxvolfnorm}
  \left\| \hat V^{-1} \right\|_F \leqslant \sqrt{r \left( N - r + 1 \right)}.
\end{equation}
Unfortunately, finding maximum volume submatrix is NP-hard \cite{NPhard}. In practice, instead of it, the so-called dominant submatrices are used, which can be found using $\operatorname{maxvol}$ \cite{maxvol} algorithm. However, even in this case, it is possible to prove that estimates \eqref{eq:maxvol2norm}-\eqref{eq:maxvolfnorm} can be achieved in time linear in $N$ only up to a constant $c > 1$ \cite{dominant}. Here we show that the same estimates can be achieved using a greedy selection of a subset of columns in just $O \left( Nr^2 \right)$ operations. It is worth noting that a simple greedy volume maximization (like $QR$ decomposition with column pivoting) cannot guarantee such estimates, and the difference from \eqref{eq:maxvol2norm} can be about $2^r$ times \cite{rrqr}.

Although we achieve good spectral and Frobenius norm estimates, let us note in advance that these bounds does not guarantee that the selected submatrix $\hat V$ has close to maximum volume.

\begin{theorem}\label{th:pseudonorms}
Let $\hat V \in \mathbb{C}^{r \times N}$. Then in $O(Nr^2)$ operations it is possible to find a submatrix $\hat V \in \mathbb{C}^{r \times r}$ of matrix $V$, such that
\begin{equation}\label{eq:ufres}
  \left\| \hat V^{-1} \right\|_F \leqslant \sqrt{r \left( N - r + 1 \right)}
\end{equation}
and
\begin{equation}\label{eq:u2res}
  \left\| \hat V^{-1} \right\|_2 \leqslant \sqrt{1 + r \left( N - r \right)}.
\end{equation}
\end{theorem}
\begin{proof}
The algorithm that we are going to construct corresponds to choosing $A = I$ in Theorem \ref{th:fastbest} with $V$ being used in place of right singular vectors of $Z$. In this case, the error of the column approximation of $A = I$ directly determines the value of $\left\| \hat V^{-1} \right\|_F$ in the corresponding columns: this approach is used to construct lower bounds for column approximation accuracy \cite{melower}. Here we directly write down the algorithm corresponding to it and prove that it guarantees \eqref{eq:ufres}. And the estimate for the spectral norm \eqref{eq:u2res} directly follows from the estimate for the Frobenius norm:
\[
\begin{aligned}
  \left\| \hat V^{-1} \right\|_2^2 & = \left\| \hat V^{-1} \right\|_F^2 - \sum\limits_{k=2}^r \sigma_k^2 \left( \hat V^{-1} \right) \\
  & \leqslant \left\| \hat V^{-1} \right\|_F^2 - (r-1) \sigma_r^2 \left( \hat V^{-1} \right) \\
  & \leqslant \left\| \hat V^{-1} \right\|_F^2 - (r-1) / \left\| \hat V \right\|_2^2 \\
  & \leqslant \left\| \hat V^{-1} \right\|_F^2 - (r-1) / \left\| V \right\|_2^2 \\
  & = \left\| \hat V^{-1} \right\|_F^2 - r + 1 \\
  & \leqslant r \left( N - r + 1 \right) - r + 1 \\
  & = 1 + r \left( N - r \right).
\end{aligned}
\]

So, let us say that we are currently at the $k$-th step, and we add the $k+1$-st column. In order to simplify its selection, we will use rotations of the matrix $V$ from the left so that after $k$ steps the matrix $\hat V$ has the form
\begin{equation}\label{eq:hatv10}
  \hat V = \left[ {\begin{array}{*{20}{c}}
   {\hat V_1}  \\ 
   {0}  \\ 
\end{array} } \right] \in \mathbb{C}^{r \times k},
\end{equation}
where $\hat V_1 \in \mathbb{C}^{k \times k}$ is a square matrix (which can also be easily made right triangular). This form can be achieved, for example, after $k$ steps of modified Gram-Schmidt, and therefore after $r$ steps we get the total complexity $O \left( N r^2 \right)$.

The partition \eqref{eq:hatv10} corresponds to the partition of the matrix $V$ as
\[
  V = \left[ {\begin{array}{*{20}{c}}
   {V_1}  \\ 
   {V_2}  \\ 
\end{array} } \right] \in \mathbb{C}^{r \times N}, \quad V_1 \in \mathbb{C}^{k \times N}, \; V_2 \in \mathbb{C}^{\left(r - k \right) \times N}.
\]

Let the column $V_j = \left[ {\begin{array}{*{20}{c}}
   {V_{1,j}}  \\ 
   {V_{2,j}}  \\ 
\end{array} } \right] \in \mathbb{C}^r$ of $V$ be added to $\hat V$. We can calculate the Frobenius norm of the pseudoinverse of the new matrix as
\begin{equation}\label{eq:extend}
\begin{aligned}
  \left\| {{{\left[ {\begin{array}{*{20}{c}}
   {\hat V} & {{V_j}}  \\ 
\end{array} } \right]}^ + }} \right\|_F^2 & = \left\| {{{\left[ {\begin{array}{*{20}{c}}
   {{{\hat V}_1}} & {{V_{1,j}}}  \\ 
   0 & {{V_{2,j}}}  \\ 
\end{array} } \right]}^ + }} \right\|_F^2 \\
  & = \left\| {\left[ {\begin{array}{*{20}{c}}
   {\hat V_1^{ - 1}} & {\frac{1}{{\left\| {{V_{2,j}}} \right\|_2^2}}\hat V_1^{ - 1}{V_{1,j}}V_{2,j}^ * }  \\ 
   0 & {\frac{1}{{\left\| {{V_{2,j}}} \right\|_2^2}}V_{2,j}^ * }  \\ 
\end{array} } \right]} \right\|_F^2 \\
  & = \left\| {\hat V_1^{ - 1}} \right\|_F^2 + \frac{1}{{\left\| {{V_{2,j}}} \right\|_2^4}}\left\| {\hat V_1^{ - 1}{V_{1,j}}V_{2,j}^ * } \right\|_F^2 + \frac{1}{{\left\| {{V_{2,j}}} \right\|_2^4}}\left\| {V_{2,j}^ * } \right\|_2^2 \\
  & = \left\| \hat V^+ \right\|_F^2 + \frac{{\left\| {\hat V_1^{ - 1}{V_{1,j}}} \right\|_2^2}}{{\left\| {{V_{2,j}}} \right\|_2^2}} + \frac{1}{{\left\| {{V_{2,j}}} \right\|_2^2}}.
\end{aligned}
\end{equation}

We can select column $j$ using the condition
\begin{equation}\label{eq:jhatv}
  j = \mathop{\arg\min}\limits_{j > k} \left( 1 + \left\| {\hat V_1^{ - 1}{V_{1,j}}} \right\|_2^2 \right) / \| V_{2,j} \|_2^2.
\end{equation}
This choice costs $O \left( N r \right)$ at each step, assuming $\hat V_1^{-1} V_1$ is known. Its update at step $k$ requires $O \left(Nk \right)$ operations, since $\hat V_1$ and $V_1$ change only in one row and column (so the corresponding updates are of rank 1). Thus, the total computational complexity of searching for $r \times r$ submatrix is $O \left( N r^2 \right)$.

Using \eqref{eq:avg} with
\[
  \sum\limits_{j > k} \left\| \hat V_1^{ - 1} V_{1,j} \right\|_2^2 = \left\| \hat V_1^{-1} V_1 \right\|_F^2 - \left\| \hat V_1^{-1} \hat V_1 \right\|_F^2 = \left\| \hat V_1^{-1} \right\|_F^2 - k
\]
and
\[
  \sum\limits_{j > k} \| V_{2,j} \|_2^2 = \| V_{2,j} \|_F^2 = r-k,
\]
we get
\begin{equation}\label{eq:minj}
  \min\limits_j \left( 1 + \left\| {\hat V_1^{ - 1}{V_{1,j}}} \right\|_F^2 \right) / \| V_{2,j} \|_2^2 \leqslant \frac{1 + \mathbb{E}_{j > k} \left\| {\hat V_1^{ - 1}{V_{1,j}}} \right\|_F^2}{\mathbb{E}_{j > k} \left\| V_{2,j} \right\|_2^2} \leqslant \frac{N-2k+\left\| \hat V_1^{-1} \right\|_F^2}{r-k} = \frac{N-2k+\left\| \hat V^+ \right\|_F^2}{r-k}.
\end{equation}

We can now prove \eqref{eq:ufres} by induction, using the following proposition for $k$ columns:
\begin{equation}\label{eq:indhyp}
  \left\| \hat V^+ \right\|_F^2 \leqslant k \frac{N-k+1}{r-k+1}.
\end{equation}
If $k = 0$, we replace the numerator in \eqref{eq:jhatv} with 1 (as it was undefined, when $V_1$ is empty), which leads to choosing column $V_j$ with the highest norm and gives
\[
  \left\| V_{j}^+ \right\|_F^2 \leqslant \frac{1}{\mathbb{E}_j \left\| V_j \right\|_2^2} = \frac{N}{\left\| V \right\|_F^2} = \frac{N}{r}.
\]
Therefore, we have just proved the base case.

When moving to $k+1$ columns, combining \eqref{eq:extend}, \eqref{eq:minj} and \eqref{eq:indhyp} leads to:
\[
\begin{aligned}
  \left\| {{{\left[ {\begin{array}{*{20}{c}}
   {\hat V} & {{V_j}}  \\ 
\end{array} } \right]}^ + }} \right\|_F^2 & \leqslant \left\| \hat V^+ \right\|_F^2 + \frac{N-2k + \left\| \hat V^+ \right\|_F^2}{r-k} \\
  & = \frac{N-2k + \left\| \hat V^+ \right\|_F^2 \left( r - k + 1 \right)}{r-k} \\
  & \leqslant \frac{N-2k + k \frac{N-k+1}{r-k+1} \left( r - k + 1 \right)}{r-k} \\
  & = \frac{N-2k + k \left( N-k+1 \right)}{r-k} \\
  & = \left( k + 1 \right) \frac{N - \left( k+1 \right)}{r - \left( k + 1 \right) + 1}.
\end{aligned}
\]
And, finally, after $r$ steps, we get
\[
  \left\| \hat V^{-1} \right\|_F^2 \leqslant r \left( N - r + 1 \right).
\]
\end{proof}

The corresponding algorithm, which finds $r \times r$ highly nondegenerate submatrix is denoted by Algorithm \ref{alg:nondeg}.

\begin{algorithm}[!ht]
\caption{}
\label{alg:nondeg}
\begin{algorithmic}[1]
\Require{Orthonormal rows $V \in \mathbb{C}^{r \times N}$, $VV^* = I$.}
\Ensure{Rows $V$ are permuted in such a way that $\hat V = V_{1:r,1:r}$ is a submatrix, satisfying Theorem \ref{th:pseudonorms}.}
  \State $l := 0_{N}$
  \For{$k := 1$ \To $r$}
    \State $j = \mathop{\arg\min}\limits_{j \geqslant k} \left( 1 + l_j \right) / \| V_{k:r,j} \|_2^2$
    \State Swap $k$-th and $j$-th columns of $V$
    \State $d_{1:N} := V_{k:r,k}^* V_{k:r,1:N} / \| V_{k:r,k} \|_2$
    \For{$j := 1$ \To $N$}
      \State $l_j := l_j + \left| d_j \right|^2$
    \EndFor
    \State $v := V_{k:r,k}$
    \State $v_1 := v_1 + e^{i \arg v_1} \left\| v \right\|_2$
    \State $v := v / \left\| v \right\|_2$
    \State $V_{k:r,1:N} = V_{k:r,1:N} - 2 vv^* V_{k:r,1:N}$
  \EndFor
\end{algorithmic}
\end{algorithm}

\section{Examples}

Here we check that the suggested Algorithm \ref{alg:CCA} indeed allows reaching high-quality cross approximations. To do that, we will look at a few known examples, where other column selection algorithms may fail: Kahan matrix \cite{Kahan} and an example from \cite{CCAiter}. The latter is the following $5 \times 4$ matrix:
\begin{equation}\label{eq:example}
A = \left[ {\begin{array}{*{20}{c}}
   1 & 1 & 1 & 0  \\ 
   1 & 1 & {1 + \varepsilon } & 0  \\ 
   1 & 0 & 0 & {1 + \varepsilon }  \\ 
   1 & 0 & 0 & 1  \\ 
   0 & 0 & 0 & 1  \\ 
\end{array} } \right]
\end{equation}
The optimal subset of one column here is the 1st column, while the optimal subset of two columns includes the 2nd and the 4th columns. Consequently, greedy approach \cite{GreedyCCA} fails to select optimal columns for $r = 2$, since it selects column 1 on the first step.

On the other hand, when Algorithm \ref{alg:CCA} with $Z$ from truncated SVD is applied to the matrix \eqref{eq:example} (for instance, with $\varepsilon = 0.001$), it first selects the 4th and then the 2nd column (which are optimal), leading to the errors
\[
\begin{aligned}
  \left\| A - CW \right\|_F & = 0.835 \ldots \\
  \left\| A - CC^+ A \right\|_F & = 0.816 \ldots \\
  \left\| A - A_r \right\|_F & = 0.629 \ldots 
\end{aligned}
\]
Our algorithm is similar to greedy approach \cite{GreedyCCA} in the sense that it also only adds columns. The main difference is that Algorithm \ref{alg:CCA} requires the knowledge of the final number of columns $r$. This allows it to select columns, which will be ``good'' with others, selected later.

Kahan matrix is
\begin{equation}\label{eq:Kahan}
A = {\small \left[ {\begin{array}{*{20}{c}}
   1 & 0 &  \cdots  & 0  \\ 
   0 & s &  \ddots  &  \vdots   \\ 
    \vdots  &  \ddots  &  \ddots  & 0  \\ 
   0 &  \cdots  & 0 & {{s^r}}  \\ 
\end{array} } \right]\left[ {\begin{array}{*{20}{c}}
   1 & { - c} &  \cdots  & { - c}  \\ 
   0 & 1 &  \ddots  &  \vdots   \\ 
    \vdots  &  \ddots  &  \ddots  & { - c}  \\ 
   0 &  \cdots  & 0 & 1  \\ 
\end{array} } \right]} \in \mathbb{R}^{(r+1) \times (r+1)}, \quad s^2 + c^2 = 1.
\end{equation}
Here $c$ should be selected close to 1 and $s$ close to 0. This example was constructed to show that incomplete QR decomposition with column pivoting may select columns, which are far from optimal, leading to error, exponentially far from SVD in terms of rank $r$. Namely, the standard column pivoting selects the first $r$ columns, leading to
\[
\begin{aligned}
  \left\| A - CC^+A \right\|_F & = \left\| A - QR \right\|_F = s^{N-1} \geqslant c (1+c)^{r-1} \left\| A - A_r \right\|_F.
\end{aligned}
\]
On the other hand, suggested algorithm always selects columns, which result in approximation error being close to SVD, see Figure \ref{new-fig}. Here the chosen columns were also always optimal: it always selected all columns, except the first one, when $r > 1$.

\begin{figure}[ht]
\centering
\includegraphics[width=0.7\columnwidth]{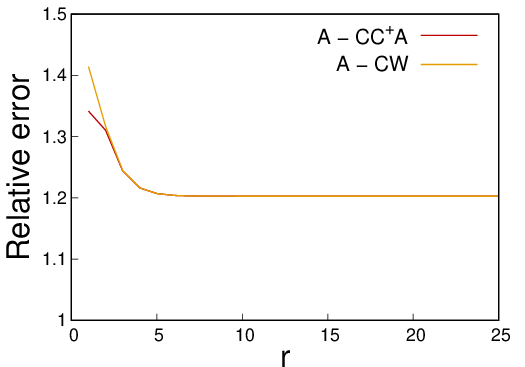}
\caption{
Performance of Algorithm \ref{alg:CCA} for Kahan matrices \eqref{eq:Kahan}. The relative error is computed by dividing the column approximation error by the SVD error. Size of the matrix is always equal to $r+1$ for each $r$. Parameter $c = 0.8$.
}
\label{new-fig}
\end{figure}

\section{Conclusion}

We have found that to construct column and cross approximations of rank $r$ using $r$ rows and columns with an optimal error coefficient, it is sufficient to have an arbitrary approximation of $Z$ of rank $r$ of sufficiently high accuracy. Then after that, in $O \left( MNr \right)$ operations, approximations close to optimal can be constructed both in the Frobenius norm and in the spectral norm.

Using an analog of this column subset selection algorithm applied to orthonormal rows $V \in \mathbb{C}^{r\times N}$, it is possible in $O\left(N r^2\right)$ operations to achieve the same estimates as for submatrices of maximum volume. The resulting submatrix can be used directly instead of searching for a ``good'' submatrix using the $\operatorname{maxvol}$ \cite{maxvol} algorithm.


\bibliographystyle{unsrt}

\bibliography{engbibfile}

\end{document}